\documentclass[11pt]{article}
\usepackage[T1]{fontenc}
\usepackage[margin=1in]{geometry}
\usepackage{amsmath,amssymb,amsthm}
\usepackage{microtype}
\usepackage[hidelinks]{hyperref}
\hypersetup{pdftitle={Deletion thresholds and exponential examples for complete sequences},pdfauthor={Jesse Geneson},pdfsubject={},pdfkeywords={}}
\newtheorem{theorem}{Theorem}
\newtheorem{lemma}[theorem]{Lemma}
\newtheorem{proposition}[theorem]{Proposition}
\newtheorem{corollary}[theorem]{Corollary}

\newcommand{\FS}{\mathop{\rm FS}}
\newcommand{\Z}{\mathbb Z}

\title{Deletion thresholds and exponential examples\\for complete sequences}
\author{Jesse Geneson}
\date{}
\begin{document}
\maketitle

\begin{abstract}
We prove that the pairs of integers $0\le m<n$ for which a nondecreasing
integer sequence can remain complete after every deletion of $m$ terms
and become incomplete after every deletion of $n$ terms are exactly
those with $m\le1$. Here a sequence is complete if every sufficiently
large integer is a finite sum of terms with distinct indices. This answers
Erd\H{o}s Problem~348, posed by Graham and later included in the book
of Erd\H{o}s and Graham.
The proof uses a central-interval theorem: if a complete nondecreasing
positive integer sequence $(a_i)$ has prefix sums $S_j$ with
$S_j-a_{j+1}\to\infty$, then each sufficiently long prefix represents
every integer from any fixed completeness threshold $T$ to $S_j-T$.
We also refute Graham's conjecture, later repeated by Erd\H{o}s and
Graham, that $(\lfloor t\gamma^n\rfloor)_{n\ge1}$ is complete for every
$t>0$ and $1<\gamma<(1+\sqrt5)/2$. We obtain the counterexample by
combining Dubickas's fractional-part theorem with an elementary sign
adjustment. For a common base in this range, we further construct two
such sequences whose interleaving is incomplete and whose coefficient
ratio is not a rational multiple of any integer power of the base.
\end{abstract}

\section{Introduction}

For a sequence $A=(a_i)_{i\ge1}$ of integers, let $\FS(A)$ be the set of
sums of finite selections of its terms, including the empty sum zero.
Each occurrence may be selected at most once. Repeated values are
separate occurrences, and a deletion removes an occurrence. We say that
$A$ is \emph{complete} if $\FS(A)$ contains every sufficiently large
integer. Thus completeness in this paper means eventual completeness.
These are the conventions in Erd\H{o}s and Graham~\cite[p.~54]{EG}.

Graham \cite[Question~3, p.~34]{Graham71} asked which pairs $0\le m<n$
admit a sequence that remains complete after every deletion of $m$
terms and becomes incomplete after every deletion of $n$ terms.
Erd\H{o}s and Graham repeated the question in their 1980 book
\cite[p.~57]{EG}; it is now recorded as Erd\H{o}s Problem~348
\cite{Bloom348}.
The powers of two give the cases $m=0$. The Fibonacci sequence with
two initial ones gives the cases $m=1$ by results of Brown \cite{Brown}
and Graham \cite{GrahamFib64}. We show that these are the only possibilities.

\begin{theorem}\label{thm:classification}
For integers $0\le m<n$, there is a nondecreasing integer sequence that
is complete after every deletion of $m$ occurrences and incomplete after
every deletion of $n$ occurrences if and only if $m\le1$.
\end{theorem}

Graham \cite{GrahamFib64} also constructed a sequence that remains
complete after every finite deletion but becomes incomplete after every
infinite deletion. Theorem~\ref{thm:classification} addresses the
contrasting case in which every deletion of some fixed finite size
destroys completeness.

For $x\in\mathbb R$, let $\|x\|$ be the distance from $x$ to the nearest
integer. Fan \cite[Corollary~1.2]{Fan} proved that the increasing sequence of
elements of a set $A$ of positive integers remains complete after every
finite deletion if $|A\cap(2^k,2^{k+1}]|\ge5$ for every sufficiently
large integer $k$ and $\sum_{a\in A}\|a\theta\|=\infty$ for every
$\theta\in\mathbb R\setminus\mathbb Z$.

The distinction between representing every positive integer and
representing every sufficiently large integer is essential here.
As recorded in the discussion of Problem~348 \cite{Bloom348},
van Doorn ruled out $m\ge2$ under the former requirement.
Our result allows a finite set of missing integers, with no restriction
on its size or on how the completeness threshold changes after deletion.
In fact, we prove that every finite deletion leaving a complete sequence
can be enlarged by one term whenever the original nondecreasing positive sequence
tolerates every two-term deletion. This is an existence statement for
larger deletions, not a claim that every finite deletion preserves
completeness.

For representation of every positive integer, Brown \cite{Brown} gave
a criterion that compares each term with the sum of its predecessors.
We recall the short interval induction in Section~\ref{sec:examples}.
The main step for eventual completeness is the central-interval theorem
(Theorem~\ref{thm:central}). For a nondecreasing positive integer sequence
$(b_i)$, let $S_N=\sum_{i=1}^N b_i$ and call $S_N-b_{N+1}$ its
\emph{prefix slack}. Fix an integer $T\ge1$, and suppose that every
integer at least $T$ is represented and that the prefix slack tends to
infinity. We prove
that, for every sufficiently large $N$, all integers between $T$ and
$S_N-T$ are represented using only the first $N$ terms. The theorem
retains the original threshold $T$, so the two omitted end intervals
have lengths independent of $N$.
For a sequence with these central intervals that becomes incomplete
after any single deletion, we show that deleting a suitable pair creates gaps longer
than any prescribed bound, arbitrarily far out in the subset sums.
The pair may depend on the prescribed bound. The full proof allows
repeated terms.
For unbounded sequences, the elementary implication from completeness
after every single deletion to divergent prefix slack is also recorded
in the partial SciNet report \cite{SciNet}. We include a proof that
also covers bounded sequences.

We also consider the question of which positive $t,\gamma$ make
$(\lfloor t\gamma^n\rfloor)_{n\ge1}$ complete.
Graham \cite{Graham64} determined the complete sequences in the range
$0<t<1$ and $1<\gamma<2$ and suggested that all $t>0$ and
$1<\gamma<\varphi$ give complete sequences, where
$\varphi=(1+\sqrt5)/2$ \cite[pp.~69--70]{Graham64}.
The proposal appears again in Graham \cite[Question~2, p.~34]{Graham71}
and Erd\H{o}s and Graham \cite[p.~57]{EG}. The parameter classification
is Erd\H{o}s Problem~349 \cite{Bloom349}.
Van Doorn \cite{VD} extended Graham's classification to the range
$\gamma\ge\varphi$ and gave partial results below $\varphi$.
In particular, he proved completeness for
$1.2<\gamma\le1.3$ and $0<t\le5$ by a computer-assisted argument.
Sothanaphan posted further refinements of bounds in the discussion of
Problem~349 \cite{Bloom349}.

Dubickas \cite[p.~332]{D} exhibits a Salem number $\gamma$ for which
a nonzero real coefficient makes all the integer parts even. His
fractional-part theorem \cite[Theorem~6]{D}, which uses earlier work of
Za\"imi \cite{Zaimi} in the Salem case, does not specify the sign of its
coefficient. We apply an elementary sign adjustment, then double the
coefficient and shift the exponents to obtain arbitrarily large positive
coefficients $t$ for which every $\lfloor t\gamma^n\rfloor$ is even,
for $n\ge0$. We check $6/5<\gamma<13/10$ and give the full deduction
in Theorem~\ref{thm:main}.
The example disproves the proposed universal range below $\varphi$:
despite its growth rate, the sequence cannot represent any odd integer.
It does not classify all pairs in Problem~349. Our argument is
existential and does not compute a counterexample coefficient. Van
Doorn's result shows that every such coefficient for this base exceeds 5.

The same construction gives two sequences
$(\lfloor\alpha\gamma^n\rfloor)_{n\ge0}$ and
$(\lfloor\beta\gamma^n\rfloor)_{n\ge0}$ with a common base
$1<\gamma<\varphi$ whose coefficients $\alpha,\beta>0$ satisfy
$\beta/\alpha\ne r\gamma^k$ for all $r\in\mathbb Q$ and
$k\in\mathbb Z$, and whose terms are all even
(Corollary~\ref{cor:two-sequences}). In particular, the coefficient
ratio is irrational, neither sequence is a tail of the other, and their
interleaving is incomplete even with repeated occurrences retained.
This answers negatively the variable-base extension of the two-sequence
question in Graham \cite[Question~12, p.~36]{Graham71} and
Erd\H{o}s and Graham \cite[p.~58]{EG}. It does not resolve the original
base-$2$ question, recorded as Erd\H{o}s Problem~354 \cite{Bloom354}.

In the base-$2$ setting, Hegyv\'ari
\cite{H} proved incompleteness when the larger coefficient is a positive
integral power of 2 times the smaller coefficient and the latter is at
least 2. Hegyv\'ari later constructed pairs for which the subset sums
contain no infinite arithmetic progression \cite{H94}.
Jiang and Ma \cite{JM} studied the smaller-coefficient range $(1,2)$,
and Fang and He \cite{FH} studied excluded sums when the coefficient
ratio is an integral power of 2.
Hegyv\'ari \cite{H24} also proved that, for $1<a<2$, every nonnegative
integer is a sum of two subset sums of
$(\lfloor a2^n\rfloor)_{n\ge0}$. The two selections are independent,
so this statement allows each indexed term to be used twice.
Chen and Fang \cite{CF} and Ma and Chen \cite{MC} also studied
completeness after taking integer parts of positive scalings of integer
sequences.

In Section~\ref{sec:deletion}, we deduce the deletion obstruction from
the central-interval theorem. Section~\ref{sec:examples} verifies the
binary and Fibonacci constructions. Section~\ref{sec:central} proves
the central-interval theorem, including repeated terms.
Sections~\ref{sec:proof} and~\ref{sec:common-base} give the exponential
constructions and their proofs.

\section{The deletion classification}\label{sec:deletion}

We first deduce the exclusion of $m\ge2$ from the central-interval
theorem stated below. We then verify the two known constructions.
The proof of the central-interval theorem is given in
Section~\ref{sec:central}.

For a nondecreasing positive integer sequence $B=(b_i)_{i\ge1}$, define
\[
 S_j=\sum_{i=1}^j b_i,\qquad
 P_j=\FS(b_1,\ldots,b_j),\qquad
 \delta_j=S_j-b_{j+1},
\]
with $S_0=0$ and $P_0=\{0\}$. We call $\delta_j$ the \emph{prefix
slack}. Every $P_j$ is symmetric about $S_j/2$: a selection with sum $x$
has a complementary selection with sum $S_j-x$. Throughout the deletion
arguments, intervals denote sets of integers.

For sets $X,Y\subseteq\Z$ and an integer $c$, we write
\[
 \begin{aligned}
  X+Y &= \{x+y:x\in X,\ y\in Y\},\\
  c+X=X+c &= \{c+x:x\in X\},\\
  c-X &= \{c-x:x\in X\}.
 \end{aligned}
\]
In particular, $c+X$ denotes the translate of $X$ obtained by
adding $c$ to each of its elements.

\begin{theorem}[central-interval]\label{thm:central}
Let $B$ be a nondecreasing sequence of positive integers. Suppose that
$[T,\infty)\subseteq\FS(B)$ for an integer $T\ge1$, and that
$\delta_j\to\infty$. Then
\[
 [T,S_N-T]\subseteq P_N
\]
for every sufficiently large $N$.
\end{theorem}

The threshold in this theorem is the original completeness threshold
$T$. In particular, the two omitted end intervals have bounded length,
independent of $N$.

The unbounded case of the first assertion in the following lemma is
also recorded in the partial SciNet report~\cite{SciNet}. We include
the full proof.

\begin{lemma}\label{lem:slack}
If a nondecreasing positive integer sequence $A=(a_i)_{i\ge1}$ remains complete after
every deletion of one occurrence, then its prefix slack tends to
infinity. The slack still tends to infinity after any fixed finite
deletion.
\end{lemma}

\begin{proof}
If the sequence is bounded, it is eventually constant and the conclusion
is immediate. Suppose it is unbounded, and fix an index $i$. After
deleting $a_i$, the remaining sequence is complete. For all sufficiently
large $k\ge i$, we therefore have
\[
 a_{k+1}\le \sum_{r=1}^k a_r-a_i+1.
\]
Otherwise the integer on the right would be missing: the retained
first $k$ terms have smaller total, and every later term is larger.
These missing integers would tend to infinity. It follows that the
original slack is at least $a_i-1$ eventually. Since $a_i$ is unbounded,
the slack tends to infinity. A fixed finite deletion subtracts its fixed
total from every sufficiently late prefix slack, which proves the last
assertion.
\end{proof}

We call a complete sequence \emph{deletion-minimal} if deleting any one
of its occurrences makes it incomplete. The next lemma explains why
central intervals are useful. A long gap between tail sums must remain
a long gap after adding a bounded prefix.

\begin{lemma}\label{lem:pair}
Let $B$ be a deletion-minimal nondecreasing positive integer sequence.
Suppose that, for some fixed integer $T\ge1$, we have
$[T,S_k-T]\subseteq P_k$ for every sufficiently large $k$.
For every $C\ge0$, there are two occurrences whose deletion leaves
arbitrarily late intervals of more than $C$ consecutive missing
integers.
\end{lemma}

\begin{proof}
The sequence $B$ is unbounded. Indeed, a bounded nondecreasing sequence
is eventually constant, and deleting one occurrence sufficiently far
along its constant tail does not change its full subset-sum set.

Choose an index $i$ with $b_i>2T+C$. Choose $k\ge i$ sufficiently large
that $P_k$ contains $[T,S_k-T]$ and $S_k\ge2T$. Fix any $j>k$, and let
$Q$ be the set of subset sums of the terms after index $k$, with the
occurrence at index $j$ deleted. The set $Q$ is an unbounded subset of
the nonnegative integers and contains zero. The subset sums after
deleting only $b_j$ are $P_k+Q$.

Put $L=S_k-2T+1$. There are arbitrarily large consecutive elements
$q<q'$ of $Q$ with $q'-q>L$. To check this, suppose that all sufficiently
late consecutive gaps of $Q$ were at most $L$. The intervals
$[q+T,q+S_k-T]$, as $q$ runs through that tail of $Q$, would then meet
or be adjacent. They would cover all sufficiently large integers, so
$P_k+Q$ would contain every sufficiently large integer. This contradicts
deletion-minimality.

Let $F$ be the set of subset sums of the first $k$ terms with $b_i$
deleted. Then $F\subseteq[0,S_k-b_i]$, and the subset sums after both
deletions are $F+Q$. For consecutive $q<q'$ as above, this set misses
\[
 [q+S_k-b_i+1,q'-1].
\]
Indeed, a sum using a tail subset sum at most $q$ is at most $q+S_k-b_i$,
whereas a sum using a tail subset sum at least $q'$ is at least $q'$.
Since the gaps are integral, the displayed missing interval has length
\[
 (q'-q)-(S_k-b_i)-1\ge b_i-2T+1>C.
\]
The values of $q$ are arbitrarily large, as required.
\end{proof}

\begin{proposition}\label{prop:extension}
Let $A$ be a nondecreasing positive integer sequence that remains
complete after every deletion of two occurrences. If $D$ is any finite
set of occurrences such that $A\setminus D$ is complete, then $D$ can
be enlarged by one occurrence while retaining completeness.
Consequently, for every nonnegative integer $r$, some deletion of $r$
occurrences leaves a complete sequence.
\end{proposition}

\begin{proof}
The sequence $A$ remains complete after every single deletion: extend
that deletion to two occurrences and then restore the second occurrence.
By Lemma~\ref{lem:slack}, the slack of $B=A\setminus D$ tends to
infinity. Since $B$ is complete and its prefix slack tends to infinity,
Theorem~\ref{thm:central}, proved in Section~\ref{sec:central},
gives a fixed integer $T\ge1$ such that $[T,S_k-T]\subseteq P_k$ for every sufficiently large $k$.

Suppose that no one-occurrence enlargement of $D$ leaves a complete
sequence. Then $B$ is deletion-minimal. Let $C$ be the sum of the
occurrences in $D$. Lemma~\ref{lem:pair} supplies two occurrences
$E$ in $B$ such that $R=\FS(B\setminus E)$ has arbitrarily late
missing intervals of length greater than $C$. Restoring $D$ adds a
subset sum in $\FS(D)\subseteq[0,C]$. If
$[u,u+\ell-1]$ is missing from $R$ and $\ell>C$, then
\[
 [u+C,u+\ell-1]\cap\bigl(R+\FS(D)\bigr)=\varnothing.
\]
For every $x$ in that interval and every $d\in\FS(D)$, the integer
$x-d$ lies in $[u,u+\ell-1]$. Thus $A\setminus E$ is incomplete,
contrary to the hypothesis on pairs. This proves the extension
assertion. Starting with $D=\varnothing$ and applying it $r$ times
proves the last assertion.
\end{proof}

\begin{proof}[Exclusion of $m\ge2$ in Theorem~\ref{thm:classification}]
First suppose all terms are positive. If a sequence remains complete
after every deletion of $m\ge2$ occurrences, it remains complete after
every deletion of two occurrences: extend the pair to $m$ occurrences
and then restore the extra terms. Proposition~\ref{prop:extension}
gives a complete sequence after some deletion of $n$ occurrences,
contradicting the required universal failure after $n$ deletions.

The same conclusion holds for nondecreasing integer sequences. A bounded
nondecreasing complete integer sequence is eventually equal to a
positive constant. Deleting $n$ sufficiently late occurrences of that
constant leaves its full subset-sum set unchanged, contradicting the
required failure after every $n$ deletions.

We may therefore assume that the original sequence $A$ is unbounded.
Since $A$ is nondecreasing, it has only finitely many nonpositive
terms. Form a positive sequence $A^+$ by discarding the zeros and
replacing each negative term by its absolute value. Only finitely
many of these positive occurrences lie below any fixed bound, so
we can arrange them in nondecreasing order.

Identify each occurrence of $A^+$ with the nonzero occurrence of $A$
from which it came. For any finite set $D$ of these occurrences,
let $C_D$ be the sum of the absolute values of the negative terms
remaining in $A\setminus D$. Then
\[
 \FS(A^+\setminus D)=C_D+\FS(A\setminus D).
\]
Indeed, replacing a surviving term $-c$ by $c$ translates its two
possible contributions by $c$, since $\{0,c\}=c+\{0,-c\}$.
Applying this identity to all surviving negative terms gives the
displayed equality; discarding zeros has no effect on subset sums.
Since translation by a fixed integer preserves eventual
completeness, $A^+\setminus D$ is complete if and only if
$A\setminus D$ is complete.

Thus $A^+$ is complete after every deletion of two occurrences and
incomplete after every deletion of $n$ occurrences, contradicting
the positive case.
\end{proof}

\section{The two existence constructions}\label{sec:examples}

We verify the powers-of-two and Fibonacci examples recorded by
Erd\H{o}s and Graham~\cite[p.~57]{EG}. Graham \cite{GrahamFib64}
proved incompleteness after two Fibonacci deletions and cited Brown
\cite{Brown} for completeness after one deletion. We include proofs to
fix the role of the two initial ones and to keep track of repeated
occurrences.

\begin{proof}[Existence for $m=0$]
The sequence $1,2,4,8,\ldots$ represents every nonnegative integer by
binary expansion. After deleting $2^j$, no remaining subset sum is
congruent to $2^j$ modulo $2^{j+1}$: all smaller terms together sum to
$2^j-1$, and every larger term is divisible by $2^{j+1}$. Thus a
one-occurrence deletion leaves infinitely many missing integers.
Further deletions cannot restore a representation, so this sequence
works for every pair $(0,n)$ with $n\ge1$.
\end{proof}

For the second example, let $F_1=F_2=1$ and
$F_{k+2}=F_{k+1}+F_k$. We use the identity
\[
 \sum_{r=1}^k F_r=F_{k+2}-1,
\]
which follows by induction on $k$. We also use the standard interval
criterion of Brown \cite{Brown}, whose short induction we recall:
if a nondecreasing positive sequence starts with one, and each following
term is at most one plus the sum of its predecessors, then each prefix
represents every integer between zero and its total. Indeed, if
$[0,S]$ is represented and the next term is $c\le S+1$, the two
represented intervals $[0,S]$ and $[c,c+S]$ meet or are adjacent.

\begin{proof}[Existence for $m=1$]
Delete any occurrence $F_j$. The remaining sequence still starts with
one. For an original index $s>j$, the sum of the retained terms before
$F_s$ is $F_{s+1}-1-F_j$. The required interval inequality is
\[
 F_s\le F_{s+1}-F_j,
\]
which is equivalent to $F_{s-1}\ge F_j$ and holds because $s-1\ge j$.
Before the deletion the interval inequality is unchanged. The interval
induction therefore shows that every nonnegative integer is represented
after any one-occurrence deletion.

Now delete $F_i,F_j$ with $i<j$. For $t\ge0$ and $1\le d\le F_i$, put
$x_t(d)=F_{j+2t+1}-d$. We prove by induction on $t$ that these integers
are missing. For $t=0$, the number $x_0(d)$ is smaller than the first
term $F_{j+1}$ after the later deletion, whereas the available terms
before that point have total
\[
 F_{j+1}-1-F_i<x_0(d).
\]
This proves the base case. For $t>0$, set $s=j+2t$. A representation of
$x_t(d)<F_{s+1}$ could use no index greater than $s$. If it also omitted
$F_s$, its total would be at most
\[
 F_{s+1}-1-F_i-F_j<x_t(d).
\]
Thus it would have to use $F_s$. Subtracting $F_s$ would represent
$F_{s-1}-d=x_{t-1}(d)$, contrary to the induction hypothesis.
This completes the induction. These missing integers tend to infinity,
so every pair deletion destroys completeness. Any deletion of $n\ge2$
occurrences contains a pair and also destroys completeness. The
Fibonacci sequence therefore works for every pair $(1,n)$ with $n\ge2$.
\end{proof}

\section{Central intervals from diverging slack}\label{sec:central}

We now prove Theorem~\ref{thm:central}. The main observation is that
every missing integer in a central interval lies just below an earlier
prefix total. For sufficiently long prefixes, its distance below that
total belongs to a fixed finite set of globally missing integers.

\subsection{Locating a missing integer}

In the following, we assume that $B$ is a nondecreasing sequence of positive integers. Retain the notation $S_j,P_j,\delta_j$ from Section~\ref{sec:deletion}, and put
$P=\FS(B)$. We assume that $\delta_j\to\infty$. Fix $T\ge1$ with $[T,\infty)\subseteq P$.

\begin{lemma}\label{lem:cut}
If $x\in[T,S_N-T]$ is missing from $P_N$, then
\[
 x=S_j-u\qquad\text{for some }j<N\text{ and }0\le u<T.
\]
\end{lemma}

\begin{proof}
Choose the least $i\le N$ with $S_i\ge x+T$. Then
$y=S_i-x$ satisfies $T\le y<b_i+T$. If $y<b_i$, a global
representation of $y$ uses only terms before index $i$, since every
term at index at least $i$ is at least $b_i$. Its complement in the
first $i$ terms represents $x$, a contradiction. Therefore $y\ge b_i$,
and hence
$x\le S_{i-1}<x+T$. Take $j=i-1$ and $u=S_j-x$.
\end{proof}

A \emph{cut at index $j$} divides the sequence into its first $j$
occurrences and the occurrences after them. We call $j$ the
\emph{cut index} and $S_j$ the \emph{cut total}. Thus
Lemma~\ref{lem:cut} places each missing integer $x$ just below
some cut total $S_j$, with $0\le S_j-x<T$.

Let $H$ be the set of nonnegative integers missing from $P$. This is a
finite subset of $\{1,\ldots,T-1\}$.

We first dispose of bounded sequences. If $B$ is bounded, it is
eventually constant, with value $a$. Let its prefix before this constant
tail have total $C$. A prefix containing $v$ copies from the constant
tail represents every globally represented integer $x\le va$.
Indeed, a global representation has the form $x=u+ra$, with $u$ a sum
from the fixed prefix and $r\ge0$, so $r\le v$. By reflection, the
same prefix represents $x$ whenever $C+va-x$ is globally represented
and at most $va$. On the central interval $[T,C+va-T]$, the first
argument supplies every $x\le va$, and reflection supplies every
$x\ge C$. For $va\ge C$, these ranges cover the central interval.
This proves the theorem in the bounded case.

Henceforth assume that $B$ is unbounded, so $b_j\to\infty$.
Suppose that $x\in[T,S_N-T]$ is missing from $P_N$. Since $x\ge T$,
it has a representation using terms of $B$. Any such representation
must use a term after index $N$, since otherwise $x$ would belong
to $P_N$. Every term after index $N$ is at least $b_{N+1}$,
and all terms are positive, so $x\ge b_{N+1}$.

By Lemma~\ref{lem:cut}, we can write $x=S_j-u$ for some $j<N$
and $0\le u<T$. Thus
\[
 S_j=x+u\ge x\ge b_{N+1}.
\]
Since $b_{N+1}\to\infty$, the corresponding indices $j$ must tend
to infinity whenever such missing integers occur for arbitrarily
large $N$. Indeed, for any fixed index $J'$, sufficiently large $N$
satisfy $b_{N+1}>S_{J'}$, which forces $j>J'$.

For each globally represented integer $w$ with $0\le w<T$,
choose a representation of $w$. Each chosen representation uses
only finitely many terms, and there are only finitely many such
integers $w$. Hence all these representations lie in one fixed
prefix $P_J$. In particular, every globally represented integer
below $T$ belongs to $P_j$ whenever $j\ge J$.

For sufficiently large $N$, the corresponding cut index satisfies
$j\ge J$. Therefore $u$ must belong to $H$: otherwise a
representation of $u$ in the first $j$ terms could be complemented
to represent $S_j-u=x$, contradicting $x\notin P_N$.
If $H$ is empty, this rules out missing integers in
$[T,S_N-T]$ for all sufficiently large $N$, proving the theorem
in that case. We may therefore assume that $H$ is nonempty, and put
\[
 h=\max H,\qquad q=|H|.
\]

\subsection{Families of missing sums}

We call $(j,N,E)$ a \emph{state} if $0\le j<N$, if
$\varnothing\ne E\subseteq H$, and if all the integers in $S_j-E$
are missing from $P_N$. Here $j$ is the cut index, while $N$
specifies the prefix in which these integers are missing.
The elements of $E$ are the state's \emph{labels}: each
$e\in E$ gives a missing integer $S_j-e$, at distance $e$
below the cut total.

The occurrences at indices $j+1,\ldots,N$ are the state's
\emph{available terms}. These lie after the cut but within the
prefix of length $N$. A state is \emph{constant} if all its
available terms have the same value, and is \emph{distinct}
if at least two different values occur among them.

We will apply the constructions below to sequences of states whose
cut indices tend to infinity. Since $b_j\to\infty$ and $S_j\to\infty$,
any fixed lower bounds on term values and prefix totals hold at all
sufficiently large indices. In each construction, we will verify that
the output cut indices also tend to infinity, so that these bounds
remain available at the output cuts.

Every distinct state is regarded as normalized and is left unchanged
by normalization. A constant state with $j=0$ is also normalized.
For a constant state with $j\ge1$, write $a=b_{j+1}$ for its available
value. If $b_j<a$, the state is already normalized. Otherwise $b_j=a$.
Let $p$ be the index immediately before the entire block of occurrences
equal to $a$. Replace the state by $(p,p+1,E)$. This is valid: a
representation of $S_p-e$ in $P_{p+1}$, together with the $j-p$
additional copies of $a$ at positions $p+2,\ldots,j+1$, would represent
$S_j-e$. These are distinct occurrences, and $N\ge j+1$ supplies all
of them. After normalization we have
\[
 j=0\quad\text{or}\quad b_j<b_{j+1}
 \quad\text{in every constant state.}
\]
Normalization preserves the labels. It also preserves divergence of
the cut index: when the original cuts tend to infinity, their available
values tend to infinity, so the first index at which each such value
occurs tends to infinity. Thus all later uses of normalization along
such a sequence of states may assume that the normalized cut is positive.

Since $b_r\to\infty$, we may fix once and for all an index $J_1$,
depending only on $B$ and $T$, such that $b_r\ge T$ for every $r>J_1$.

\begin{lemma}\label{lem:close}
Let $(j,N,E)$ be a state with $j\ge J_1$. Put $a=b_{j+1}$, and let $b$
be the first term value strictly greater than $a$ occurring after
index $j$. Then
\[
 1\le b-a\le\min E\le h.
\]
The threshold $J_1$ depends only on $B$ and $T$, not on the state, and
the term of value $b$ need not belong to the certifying prefix $P_N$.
\end{lemma}

\begin{proof}
The value $b$ exists because the sequence is unbounded, and $b>a$ by
definition, so $b-a\ge1$. Also $\min E\le h$ because $E\subseteq H$.
It remains to prove that $b-a\le e$ for every $e\in E$.

Fix $e\in E$ and suppose that $b>a+e$. Since $j\ge J_1$ we have
$a=b_{j+1}\ge T$, and $e\ge1$, so $a+e\ge T$ and $a+e$ has a global
representation. Because $B$ is nondecreasing with $b_{j+1}=a$, every
term of $B$ has value at most $a$ or at least $b$, and the latter
exceeds $a+e$; hence every term of the representation has value at
most $a$. The representation cannot use an occurrence of $a$, since
removing that occurrence would represent the globally missing integer
$e$. Thus all its terms have value smaller than $a$, and therefore lie
before the cut $j$. Complement this representation in the first $j$
terms and add the occurrence at index $j+1$, which is available
because $N>j$. The resulting sum is $S_j-e$, contrary to the state
condition.
\end{proof}

We will repeatedly subtract an available term from each missing sum.
The resulting sums must be missing from the earlier prefix: adding
back that one occurrence would otherwise give a forbidden
representation. The following observation ensures that the resulting
missing sums again share one cut.

\begin{lemma}\label{lem:common}
Suppose $j$ tends to infinity and $R_j$ is a nonempty set of
integers missing from $P_j$, with
\[
 \min R_j\longrightarrow\infty,\qquad
 \max R_j\le S_j-T,\qquad
 \max R_j-\min R_j\le2h.
\]
Then, for all sufficiently large $j$, there are an index $l<j$
and a nonempty set $E'\subseteq H$ such that $R_j=S_l-E'$.
Moreover, $l\to\infty$.
\end{lemma}

\begin{proof}
Choose $J$ such that $P_J$ contains every globally represented
integer below $T$ and $b_{J+1}>3h$. This is possible because
there are only finitely many such integers and $b_r\to\infty$.

Take $j$ sufficiently large that $\min R_j>\max\{T,S_J\}$.
For each $x\in R_j$, Lemma~\ref{lem:cut} gives
$x=S_{l_x}-u_x$ with $l_x<j$ and $0\le u_x<T$.
Since $S_{l_x}=x+u_x>S_J$, we have $l_x>J$.
If $u_x\notin H$, then $u_x\in P_J\subseteq P_{l_x}$.
Complementing a representation of $u_x$ within the first $l_x$
terms would represent $x$, contradicting $x\notin P_j$.
Thus every $u_x$ belongs to $H\subseteq[1,h]$.

For any $x,y\in R_j$, we therefore have
\[
 |S_{l_x}-S_{l_y}|
 \le |x-y|+|u_x-u_y|
 \le 2h+h=3h.
\]
But if $l_x<l_y$, then
\[
 S_{l_y}-S_{l_x}
 =\sum_{r=l_x+1}^{l_y}b_r
 \ge b_{l_x+1}>3h,
\]
a contradiction. Hence all the indices $l_x$ equal one index
$l<j$. Taking $E'=\{u_x:x\in R_j\}$ gives
$R_j=S_l-E'$ with $\varnothing\ne E'\subseteq H$.
Finally, $S_l\ge\min R_j\to\infty$, so $l\to\infty$.
\end{proof}

\subsection{The two transitions}

We describe two transitions between normalized states. Their purpose
is to increase either the number of labels or their least value.

\paragraph{A distinct state increases the number of labels.}
Let $a=b_{j+1}$, and let $b$ be the first larger available value.
Lemma~\ref{lem:close} gives $1\le b-a\le h$. Define
\[
 R_j=\{S_j-a-e:e\in E\}\ \cup\ \{S_j-b-e:e\in E\}.
\]
Every member of $R_j$ is missing from $P_j$. Indeed, a
representation of $S_j-a-e$ in the first $j$ terms could be
extended by the available occurrence of $a$ to represent $S_j-e$
in $P_N$, contrary to the state condition. The same argument
applies to $b$. These added occurrences lie after index $j$,
so no occurrence is reused.

Since $\delta_j=S_j-a$ and $E\subseteq[1,h]$, we have
\[
 \begin{aligned}
  \min R_j
    &=\delta_j-(b-a)-\max E\ge\delta_j-2h,\\
  \max R_j-\min R_j
    &=(b-a)+\max E-\min E\le2h,\\
  \max R_j
    &=S_j-a-\min E\le S_j-T
       \qquad\text{for }a\ge T.
 \end{aligned}
\]
As $j\to\infty$, the first bound tends to infinity and $a\ge T$
eventually. Lemma~\ref{lem:common} therefore gives
$R_j=S_l-E'$ for some $l<j$ and $\varnothing\ne E'\subseteq H$,
with $l\to\infty$.

The first set defining $R_j$ has $|E|$ elements. The second
contains $S_j-b-\max E$, which is strictly smaller than
$S_j-a-\max E$, the smallest element of the first set.
Thus $|R_j|>|E|$. Since $R_j=S_l-E'$ gives one label for
each member of $R_j$, we obtain
\[
 |E'|=|R_j|>|E|.
\]
The new state is $(l,j,E')$. If it is constant, normalize it;
as shown above, normalization preserves its labels and the
fact that its cut index tends to infinity.

\paragraph{A constant state raises the least label or becomes distinct.}
Let $a=b_{j+1}$ be the available value of the normalized constant
state. Thus $a>b_j$. Define
\[
 R_j=\{S_j-a-e:e\in E\}=\{\delta_j-e:e\in E\}.
\]
Every member of $R_j$ is missing from $P_j$: a representation
of $S_j-a-e$ in the first $j$ terms could be extended by the
occurrence at index $j+1$ to represent $S_j-e$ in $P_N$,
contrary to the state condition.

Since $E\subseteq[1,h]$, these numbers have minimum at least
$\delta_j-h\to\infty$, spread at most $h$, and maximum at most
$S_j-T$ once $a\ge T$. Lemma~\ref{lem:common} therefore gives
$R_j=S_l-E'$ with $l<j$, $\varnothing\ne E'\subseteq H$,
and $l\to\infty$. Each description of $R_j$ gives one integer
per label, so $|E'|=|R_j|=|E|$. If the new state $(l,j,E')$
is distinct, the transition is complete.

Otherwise, the terms at indices $l+1,\ldots,j$ all equal
$b=b_j$. Put $r=j-l\ge1$, so $S_j=S_l+rb$. For each $e\in E$,
\[
 S_j-a-e=S_l-(e+a-rb).
\]
Thus the new labels are exactly $e+a-rb$ for $e\in E$, giving
\[
 E'=E+(a-rb),\qquad \min E'=\min E+a-rb.
\]

The first value greater than $b$ after index $l$ is
$a=b_{j+1}$: all the intervening terms equal $b$, and $a>b$.
Since $l\to\infty$, Lemma~\ref{lem:close} applies to the new
state and gives
\[
 a-b\le\min E'=\min E+a-rb.
\]
Rearranging, we obtain
\[
 (r-1)b\le\min E\le h.
\]
For sufficiently large $j$, we have $b=b_j>h$. If $r\ge2$,
then $(r-1)b\ge b>h$, a contradiction. Hence $r=1$, and
\[
 \min E'=\min E+(a-b)>\min E.
\]
Normalize this new constant state if necessary. As shown above,
normalization preserves its labels and the fact that its cut
index tends to infinity.

\subsection{A finite bound on the transitions}

For each normalized state $(j,N,E)$, define
\[
 \Phi=
 \begin{cases}
  (h+1)|E|+\min E,&\text{if the state is constant},\\
  (h+1)(|E|+1),&\text{if the state is distinct}.
 \end{cases}
\]
Since $1\le |E|\le q$ and $1\le\min E\le h$, this is a
positive integer at most $(h+1)(q+1)$.

Let $E'$ and $\Phi'$ denote the label set and potential after
a transition, including any required normalization.
If a constant state becomes constant, then $|E'|=|E|$
and $\min E'>\min E$, so $\Phi'>\Phi$.
If a constant state becomes distinct, then $|E'|=|E|$ and
\[
 \Phi'-\Phi=h+1-\min E\ge1.
\]
Finally, if the original state is distinct, then
$|E'|\ge|E|+1$. Regardless of the new state's type,
\[
 \Phi'\ge(h+1)|E'|+1
 \ge(h+1)(|E|+1)+1
 =\Phi+1.
\]
Thus every transition strictly increases the potential.

Suppose, for a contradiction, that for arbitrarily large $N$
there exists an integer in $[T,S_N-T]$ that is missing from $P_N$.
Lemma~\ref{lem:cut} and the observations following it give
states $(j,N,\{u\})$ with $u\in H$ and cut indices $j$
tending to infinity. Normalization preserves this property.

Since there are only finitely many possible potential values,
at least one is attained by normalized states with arbitrarily
large cut indices. Let $p$ be the largest such value, considering
all normalized states. Choose a sequence of normalized states
with potential $p$ and cut indices tending to infinity.

Apply one transition to each sufficiently late state in this
sequence. The preceding constructions show that the new states
are normalized, their cut indices still tend to infinity, and
their potentials are all greater than $p$. Since only finitely
many potential values are possible, one value $p'>p$ occurs
infinitely often among these new states. Their cut indices tend
to infinity, so $p'$ is attained at arbitrarily large cuts.
This contradicts the choice of $p$.

Therefore
\[
 [T,S_N-T]\subseteq P_N
 \qquad\text{for every sufficiently large }N.
\]
This proves Theorem~\ref{thm:central} and completes the proof
of Theorem~\ref{thm:classification}.

\section{Incomplete exponential sequences}\label{sec:proof}

A \emph{Salem number} is a real algebraic integer greater than one whose
other conjugates lie in the closed unit disk, with at least one on the
unit circle. A \emph{Pisot number} is a real algebraic integer greater
than one whose other conjugates all lie in the open unit disk.

\begin{theorem}\label{thm:main}
Let $\gamma>1$ be the Salem number with minimal polynomial
\[
 P(x)=x^{18}-x^{12}-x^{11}-x^{10}-x^9-x^8-x^7-x^6+1.
\]
Then $6/5<\gamma<13/10<\varphi$, and there are arbitrarily large positive
real numbers $t$ such that $\lfloor t\gamma^n\rfloor$ is even for every
integer $n\ge0$. In particular, these sequences are not complete.
\end{theorem}

Dubickas \cite[p.~332]{D} gives this Salem number and its minimal
polynomial. We give the positive-coefficient argument explicitly, then
extend the example to two coefficients in Section~\ref{sec:common-base}.
The statement with $n\ge0$ implies the statement with $n\ge1$, since
every term of either sequence is even.

For a real number $x$, write $\{x\}=x-\lfloor x\rfloor$ for its
fractional part. We use the following result of Dubickas.

\begin{lemma}[Dubickas {\cite[Theorem~6, p.~334]{D}}]\label{lem:dubickas}
Let $\gamma$ be a Pisot or Salem number with minimal polynomial $P$,
and suppose that $P(1)=-q$ for an integer $q\ge2$. For every
$\epsilon>0$, there is a real number $\xi\in\mathbb Q(\gamma)$ such
that
\[
 \frac1q-\epsilon<\{\xi\gamma^n\}<\frac1q+\epsilon
 \qquad(n\ge1).
\]
\end{lemma}

The coefficient in Lemma~\ref{lem:dubickas} may have either sign. The
next proposition records the elementary sign adjustment that we need.

\begin{proposition}\label{prop:positive}
Let $\gamma$ be a Pisot or Salem number with minimal polynomial $P$,
and suppose that $P(1)=-q$ for an integer $q\ge3$. There is a real
number $\eta>0$ such that
\[
 \frac3{4q}<\{\eta\gamma^n\}<\frac5{4q}
 \qquad(n\ge1).
\]
Consequently, there are arbitrarily large positive real numbers $t$
for which $\lfloor t\gamma^n\rfloor$ is even for every integer $n\ge0$.
\end{proposition}

\begin{proof}
Set $\epsilon=1/(4q(q-1))<1/q$ and choose $\xi$ by
Lemma~\ref{lem:dubickas}. Note that $\epsilon\le1/(4q)$ and
$(q-1)\epsilon=1/(4q)$. The lower bound in the lemma is positive,
so $\xi\ne0$. For each $n\ge1$, set
$a_n=\lfloor\xi\gamma^n\rfloor$ and
$e_n=\{\xi\gamma^n\}-1/q$. Then
\[
 \xi\gamma^n=a_n+\frac1q+e_n,
 \qquad a_n\in\mathbb Z,\quad |e_n|<\epsilon
 \qquad(n\ge1).
\]
Since $q\ge3$, the interval $(3/(4q),5/(4q))$ lies in $(0,1/2)$.
If $\xi>0$, set $\eta=\xi$; the residual after the integer $a_n$ is
$1/q+e_n$, which differs from $1/q$ by less than $1/(4q)$.
If $\xi<0$, set $\eta=-(q-1)\xi>0$. Then
\[
 \eta\gamma^n=-(q-1)a_n-1+\frac1q-(q-1)e_n.
\]
Here $-(q-1)a_n-1$ is an integer, and the residual
$1/q-(q-1)e_n$ differs from $1/q$ by less than $1/(4q)$.
Thus, in either case, $\eta\gamma^n$ is an integer plus a residual
in $(3/(4q),5/(4q))\subset(0,1/2)$, so the residual is its fractional
part. In either case $\eta>0$, and
\[
 \frac3{4q}<\{\eta\gamma^n\}<\frac5{4q}<\frac12
 \qquad(n\ge1).
\]
It follows that
$\lfloor2\eta\gamma^n\rfloor=2\lfloor\eta\gamma^n\rfloor$ is
even. For any integer $m\ge1$, set $t_m=2\eta\gamma^m$. Then
$\lfloor t_m\gamma^n\rfloor$ is even for every $n\ge0$, and
$t_m\to\infty$ as $m\to\infty$.
\end{proof}

\begin{proof}[Proof of Theorem~\ref{thm:main}]
Dubickas \cite[p.~332]{D} identifies the displayed polynomial as the
minimal polynomial of a Salem number $\gamma$. Since $P(1)=-5$,
Proposition~\ref{prop:positive} gives the required values of $t$.
For an exact check of the interval, direct evaluation gives
\[
 \begin{aligned}
 5^{18}P(6/5)&=-41745565065959<0,\\
 10^{18}P(13/10)&=28586401421206393129>0.
 \end{aligned}
\]
The intermediate value theorem gives a root in $(6/5,13/10)$, which
has modulus greater than one. Since $\gamma$ is the only root outside
the closed unit disk, that root is $\gamma$. Thus
$6/5<\gamma<13/10<3/2<\varphi$, where $3/2<\varphi$ since
$\sqrt5>2$. The floors are
nonnegative since $t>0$ and $\gamma>0$. Every finite sum of terms of
the resulting sequence is even, so the sequence is not
complete.
\end{proof}

Van Doorn \cite[Proposition~8]{VD} proved by a computer-assisted argument
that the sequence is complete
for $1.2<\gamma\le1.3$ and $0<t\le5$. Since the Salem number above
lies in this interval, every counterexample coefficient for this base
must exceed 5. Van Doorn indexes the sequence by $n\ge1$; adjoining
the term with $n=0$ preserves completeness. Our argument is existential:
it gives neither an explicit
coefficient nor a numerical upper bound for one.

\section{Two sequences with a common base}\label{sec:common-base}

Erd\H{o}s and Graham \cite[p.~58]{EG} ask whether the interleaving of
$(\lfloor\alpha2^n\rfloor)_{n\ge0}$ and
$(\lfloor\beta2^n\rfloor)_{n\ge0}$ is complete whenever
$\alpha,\beta>0$ and $\alpha/\beta$ is irrational. They then ask:
``What if 2 is replaced by $\gamma$ where $1<\gamma<2$?'' We retain
separate occurrences of repeated values, as in this interleaving.

Our example gives an incomplete interleaving for some such bases.
The ordinary set union is also incomplete, since discarding repeated
occurrences cannot create a representation. For a set, completeness
means that every sufficiently large integer is a sum of distinct elements.

For base 2, Hegyv\'ari \cite{H} proved incompleteness when
$\alpha\ge2$ and $\beta=2^k\alpha$ for a positive integer $k$.
Jiang and Ma \cite{JM} studied the lower coefficient range
$1<\alpha<2$ in this base-2 setting; see also the description in
Fang and He \cite{FH}. Fang and He further studied
the excluded sums when the coefficient ratio is an integral power of 2.

\begin{corollary}\label{cor:two-sequences}
There are $1<\gamma<\varphi$ and positive real numbers
$\alpha,\beta$ such that every term
of both $(\lfloor\alpha\gamma^n\rfloor)_{n\ge0}$ and
$(\lfloor\beta\gamma^n\rfloor)_{n\ge0}$ is even, and
\[
 \frac\beta\alpha\ne r\gamma^k
 \qquad(r\in\mathbb Q,\ k\in\mathbb Z).
\]
In particular, $\alpha/\beta$ is irrational, neither sequence is a
tail of the other, and their interleaving is not complete.
\end{corollary}

\begin{proof}
Let $\gamma$ be the Salem number from Theorem~\ref{thm:main}. Choose
$\eta>0$ by Proposition~\ref{prop:positive} with
$q=5$, so that
\[
 \frac3{20}<\{\eta\gamma^j\}<\frac14
 \qquad(j\ge1).
\]
Set $\alpha=2\eta\gamma$ and $\beta=\alpha(1+\gamma)$. For
$n\ge0$, the number $\alpha\gamma^n=2\eta\gamma^{n+1}$ has even
integer part. Moreover,
\[
 \frac3{10}<
 \{\eta\gamma^{n+1}\}+\{\eta\gamma^{n+2}\}<\frac12.
\]
Writing each summand $\eta\gamma^{n+1}$ and $\eta\gamma^{n+2}$ as
its integer part plus its fractional part shows that
$\beta\gamma^n=2(\eta\gamma^{n+1}+\eta\gamma^{n+2})$ also has even
integer part.

The polynomial $P$ in Theorem~\ref{thm:main} is reciprocal, since
$x^{18}P(1/x)=P(x)$, so $\gamma^{-1}$ is also a root of $P$.
Since $P$ is minimal, the map $\gamma\mapsto\gamma^{-1}$ extends
to an isomorphism from $\mathbb Q(\gamma)$ onto
$\mathbb Q(\gamma^{-1})$. These fields are equal, since each of
$\gamma$ and $\gamma^{-1}$ is the inverse of the other. Thus the
isomorphism is an automorphism of $\mathbb Q(\gamma)$. If
$\beta/\alpha=1+\gamma=r\gamma^k$ for some $r\in\mathbb Q$ and
$k\in\mathbb Z$, then $r\ne0$, and applying this automorphism gives
$1+\gamma^{-1}=r\gamma^{-k}$. Dividing the two identities yields
$\gamma=\gamma^{2k}$. Since $\gamma>1$, this forces $2k=1$, a
contradiction.

Taking $k=0$ shows that $\beta/\alpha$, and hence $\alpha/\beta$,
is irrational. If one sequence were a tail of the other, then for
some $k\ge0$ either
$|\beta-\alpha\gamma^k|\gamma^n<1$ for every $n\ge0$ or
$|\alpha-\beta\gamma^k|\gamma^n<1$ for every $n\ge0$. Letting
$n\to\infty$ would give $\beta/\alpha=\gamma^k$ or
$\beta/\alpha=\gamma^{-k}$, both excluded above. All finite sums
from the interleaving are even, so the interleaving is not complete.
\end{proof}

This corollary does not resolve the original question with base 2.
For the constructed base, incompleteness holds even with separate
occurrences retained.
The ordinary set union is also incomplete, since every available value
and therefore every finite sum is even.

\section*{Declaration of generative AI and AI-assisted technologies}
The proofs were found with the assistance of
Codex with GPT-6 Astra Ultra. The author directed separate writing and auditing teams,
read and edited the original proofs, and requested further revisions
for clarity. The author takes responsibility for the
content of the article.

\end{document}